\documentclass[10pt,oneside,reqno]{amsart}

  \usepackage[utf8]{inputenc}
  \usepackage[T1]{fontenc}
  \usepackage[russian,ngerman,english]{babel}
  \usepackage{lmodern}
  \usepackage{xcolor}

  \usepackage{hyperref}
  \usepackage{amsmath}
  \usepackage{amsthm}
  \usepackage{amsfonts}
  \usepackage{amssymb}
  \usepackage{amscd}
  \usepackage{amsbsy}
\usepackage{comment}

  \usepackage{enumerate}

  \usepackage[sort, numbers]{natbib}
  \usepackage{bibentry}

  \newtheorem{theorem}{Theorem}[section]

  \newtheorem{lemma}[theorem]{Lemma}

  \theoremstyle{definition}

  \newtheorem*{remark}{Remark}

  \newtheorem{conjecture}[theorem]{Conjecture}
  \newtheorem{example}[theorem]{Example}

  \numberwithin{equation}{section}

  \newcommand{\Z}{{\mathbb Z}}

  \newcommand{\R}{{\mathbb R}}
  \renewcommand{\C}{{\mathbb C}}

  \newcommand{\SL}{\mathrm{SL}}

  \newcommand{\bbZ}{{\mathbb Z}}
  
  \newcommand{\bbR}{{\mathbb R}}
  \newcommand{\bbC}{{\mathbb C}}

  \newcommand{\cK}{{\mathcal{K}}}

\newcommand{\M}{\mathring M^\eta}

\newcommand{\K}{\mathring{\mathcal K}^\eta}

\renewcommand{\Re}{\operatorname{Re}}
\renewcommand{\Im}{\operatorname{Im}}

\newcommand{\supp}{\operatorname{supp}}

\newcommand{\tr}{\operatorname{tr}}
\newcommand{\vl}{\overleftarrow}
\newcommand{\ol}{\overline}

\newcommand{\loc}{\textrm{loc}}
  
  \allowdisplaybreaks

\author[B.\ Eichinger]{Benjamin Eichinger}
\address{B.\ Eichinger: School of Mathematical Sciences, Lancaster University, Lancaster LA1 4YF, UK and Institute of Analysis and Scientific Computing, TU Wien, 1040~Wien, Austria}

\email{\href{mailto:b.eichinger@lancaster.ac.uk}{b.eichinger@lancaster.ac.uk}}

\author[M.\ Luki\'c]{Milivoje Luki\'c}

\address{M.\ Luki\'c: Department of Mathematics, Emory University, Atlanta, GA~30322, USA}
\email{\href{mailto:milivoje.lukic@emory.edu}{milivoje.lukic@emory.edu}}
\thanks{M. L.\ was supported in part by NSF grants DMS--2154563/2626207 and DMS--2453758/2626193.}

\author[G.\ Young]{Giorgio Young}
\address{G.\ Young: Department of Mathematics, University of Wisconsin-Madison, 480 Lincoln Dr., Madison, WI 53706, USA}
\email{\href{mailto:gfyoung@wisc.edu}{gfyoung@wisc.edu}}
\thanks{G. Y.\ acknowledges the support of the National Science Foundation through grant DMS--2303363.}

\title{Clock spacing for two-sided Jacobi matrices}

\begin{document}

\begin{abstract}
We study local eigenvalue spacing for finite truncations of a two-sided Jacobi matrix with two movable endpoints. In particular, we show that a suitable analog of clock spacing follows from a pointwise reflectionlessness condition. We obtain this as a consequence of a new scaling limit for Christoffel--Darboux kernels with a movable starting point. Without reflectionlessness, we obtain a new class of limit kernels, which combine distinct contributions from $\pm\infty$. We also show that clock spacing in the two-sided setting is a fragile phenomenon, which can be destroyed by the change of a single Jacobi coefficient; in particular, it is not merely a consequence of absolutely continuous spectrum.
\end{abstract}

	\maketitle

\section{Introduction}

The goal of this paper is to study local eigenvalue spacing for finite Jacobi matrices
\[
J_{[k,n]} = \begin{pmatrix}
b_{k+1} & a_{k+1} & & \\
a_{k+1} & \ddots & \ddots & \\
& \ddots & \ddots & a_{n-1} \\
& & a_{n-1} & b_n
\end{pmatrix}
\]
arising as finite truncations of an infinite Jacobi matrix $J$, given formally by
\[
(Ju)_n = a_{n-1} u_{n-1}+ b_n u_n + a_n u_{n+1}, \qquad u \in \ell^2(\bbZ).
\]
If one endpoint is kept fixed, say $k =0$, the asymptotic local eigenvalue spacing as $n\to \infty$ is a well-studied problem, formulated as the study of zeros of orthogonal polynomials, see for instance \cite{Szego39,ErdosTuran40,AvilaLastSimon,EichLukSimanek,EichLukWor,KuiljaarsVanlessenIMRN02,LubinskyAnnals,LevinLubinsky08,LastSimon08,BSing} and the survey article \cite{LubSigma16}.
We will study the case when $k, n$ are both movable endpoints, which was not previously studied in the literature.

We assume throughout that $a_n > 0$, $b_n \in \bbR$, and that $J$ is limit-point at both $\pm \infty$ (equivalently, that $J$ is a self-adjoint operator with a core given by the set of compactly supported sequences in $\ell^2(\bbZ)$). We will denote by $J_+$, $J_-$ the half-line Jacobi matrices obtained by restricting $J$ to $\ell^2(\{n\in\bbZ \mid n\ge 1\})$ and $\ell^2(\{n\in\bbZ \mid n \le 0\})$, and denote by $m_\pm$ the Weyl $m$-functions associated with $J_\pm$.

To study the local distribution of eigenvalues of the matrix $J_{[k,n]}$ around a point $\xi \in \bbR$, we denote these eigenvalues by $\xi_j^{[k,n]}$, so that
\[
\dots < \xi_{-1}^{[k,n]} < \xi_0^{[k,n]} < \xi \le \xi_{1}^{[k,n]} < \xi_2^{[k,n]} < \dots
\]
We will work under the assumption that 
\begin{equation}\label{mplusnormallimit}
\lim_{y\downarrow 0} m_+(\xi + i y) = \eta \in \bbC_+,
\end{equation}
i.e., the normal limit exists and is equal to some $\eta \in \bbC_+:= \{z\in\C:\ \Im z>0\}$; the set of such $\xi$ is an essential support for the a.c.\ part of $J_+$. We consider a formal eigensolution $\psi$ at $\xi \in \bbR$, i.e., a nontrivial solution of the Jacobi recursion
\begin{equation}\label{eqnJacobiRecursion}
a_{n-1} \psi_{n-1} + b_n \psi_n + a_n \psi_{n+1} = \xi \psi_n,
\end{equation}
which satisfies
\begin{equation}\label{eqnFloquetSolution}
- \frac{ \psi_1 }{ a_0 \psi_0} = \eta.
\end{equation}
This determines $\psi$ up to normalization, and we will write all formulas in a form independent of normalization. Eigensolutions obeying \eqref{eqnFloquetSolution} are well-known in spectral theory; in the scattering theory of decaying perturbations of the free Jacobi matrix, they can be recognized as multiples of Jost solutions, and in ergodic settings, they are generalizations of Floquet solutions \cite{DeiftSimon83}. 
A standard observation is that the solution $\psi$ can be interpreted as the limit of Weyl solutions at energy $\xi + iy$ as $y\downarrow 0$.

Since $\psi, \ol\psi$ are eigensolutions at $\xi$, their Wronskian 
\[
W(\ol \psi,\psi) = W_n(\ol\psi, \psi) = a_n ( \ol\psi_{n+1} \psi_n - \ol \psi_n \psi_{n+1})
\]
is independent of $n$, and a direct calculation gives $W(\ol \psi,\psi) / i = 2\lvert \psi_0 \rvert^2 \Im \eta > 0$.
We will use the scaling sequence
\begin{equation}\label{eqnScalingHalfline}
c_{[k,n]} = \frac{i}{\pi W(\ol \psi, \psi)} \sum_{j=k+1}^n \lvert \psi_j \rvert^2, \qquad 0 \le k \le n,
\end{equation}
and set $c_{[n,k]} = - c_{[k,n]}$. This will have the role of the local density of states at $\xi$ for the block $J_{[k,n]}$. We take a brief detour to relate it to macroscopic notions of density of states. 

\begin{remark}

In prominent cases of interest, orthogonal polynomials $p_n$ satisfy a limit of the form
\begin{equation}\label{eqnDensityOfStates1}
\lim_{n\to\infty} \frac 1n \sum_{j=0}^{n-1} p_j(\xi)^2 = \frac{w_\infty(\xi)}{w(\xi)},
\end{equation}
where $w(\xi) = d\mu(\xi) / d\xi \in (0,\infty)$ is the derivative of the orthogonality measure, and $w_\infty(\xi) \in (0,\infty)$ is interpreted as the density of states. In particular:
\begin{enumerate}[(a)]
\item If the orthogonality measure $\mu$ is Stahl--Totik regular, \eqref{eqnDensityOfStates1} holds under local Lebesgue point and local Szeg\H o conditions at the point $\xi$, where $w_\infty$ denotes the a.c.\ part of the equilibrium measure of $\supp\mu$ \cite{MNT,Totik2000AssympChris}.
\item For an almost periodic Jacobi matrix, \eqref{eqnDensityOfStates1} holds on an essential support of the a.c.\ part of the spectral measure, where $w_\infty$ denotes the a.c.\ part of the density of states \cite{AvilaLastSimon}.
\end{enumerate}
The following lemma explains how such results can be restated for the scaling sequence $c_{[k,n]}$. Note also the additivity relation  $c_{[k,n]} = c_{[k,l]} + c_{[l,n]}$ for $k \le l \le n$.
\end{remark}

\begin{lemma}\label{lemmaChristoffel}
If $\xi\in\bbR$ is such that \eqref{mplusnormallimit}, \eqref{eqnDensityOfStates1} hold for $J_+$, then $c_{[0,n]} / n  \to w_\infty(\xi)$ as $n\to\infty$.
\end{lemma}

We return to our main goal, which is to describe the asymptotic behavior of eigenvalues $\xi_j^{[k,n]}$ under local spectral assumptions. The behavior of the movable endpoints can be split into two cases. If both endpoints go in the same direction, say towards $+\infty$, then the eigenvalue scaling is governed by the behavior of $m_+$:

\begin{theorem}\label{thmOneSidedClock}
Assume that at the point $\xi \in \bbR$, the normal limit
\[
\lim_{y\downarrow 0} m_+(\xi + iy)
\]
exists and is in $\bbC_+$. Then for any sequences $k_l < n_l$ with $\inf k_l > -\infty$, $n_l \to \infty$ as $l\to\infty$, and
\begin{equation}\label{eqnLimsupCondition}
\limsup_{l\to \infty} \frac{c_{[0,k_l]}}{c_{[0,n_l]}}< 1,
\end{equation}
 we have
\begin{align}\label{eq:Halfline2Clock}
\lim_{l \to \infty}  c_{[k_l,n_l]} ( \xi_{j+1}^{[k_l,n_l]} - \xi_{j}^{[k_l,n_l]} ) = 1, \qquad \forall j\in \bbZ.
\end{align}
\end{theorem}

The assumption \eqref{eqnLimsupCondition} ensures that the block $[k_l,n_l]$ does not go to infinity too fast relative to its size. We regard it as an interesting open problem whether this assumption can be replaced by the weaker assumption $c_{[k_l,n_l]} \to \infty$ as $l \to\infty$ without any additional assumptions.

If the two movable endpoints move in opposite directions, the result needs a pointwise version of the reflectionless condition at $\xi$:

\begin{theorem}\label{thmTwoSidedClock1}
Assume that at the point $\xi \in \bbR$,
\begin{equation}\label{eqnReflectionlessJacobi}
 \lim_{y\downarrow 0} m_+(\xi + iy)  = \frac 1{a_0^2 \overline{ \lim\limits_{y\downarrow 0} m_-(\xi + iy) }} \in \bbC_+.
\end{equation}
Then for any sequences $k_l < n_l$ with $\sup_l k_l < \infty$, $\inf_l  n_l > -\infty$, and $\lim_{l\to\infty} (n_l - k_l) = \infty$,
 we have
\[ 
\lim_{l \to \infty}  c_{[k_l,n_l]} ( \xi_{j+1}^{[k_l,n_l]} - \xi_{j}^{[k_l,n_l]} ) = 1, \qquad \forall j\in \bbZ.
\] 
\end{theorem}

\begin{remark}
Although many of our assumptions are stated with respect to $0$ as a reference point, they are all independent of the origin. For instance, the condition \eqref{eqnReflectionlessJacobi} is equivalent for any $k\in \bbZ$ to
\begin{equation}\label{eqnReflectionlessJacobik}
 \lim_{y\downarrow 0} m_{+,k}(\xi + iy)  = \frac 1{a_k^2 \overline{ \lim\limits_{y\downarrow 0} m_{-,k}(\xi + iy) }} \in \bbC_+
\end{equation}
where $m_{\pm,k}$ correspond to half-line Jacobi matrix restrictions to $\{n\in\bbZ \mid n \ge k+1\}$ and $\{n\in\bbZ \mid n \le k\}$; this is a standard observation in the theory of reflectionless operators.
\end{remark}

We note that the conclusions of the previous two theorems imply the statement
\begin{equation}\label{eq:WeakClock}
\lim_{l \to \infty} \frac{ \xi_{j+1}^{[k_l,n_l]} - \xi_{j}^{[k_l,n_l]}}{ \xi_{1}^{[k_l,n_l]} - \xi_{0}^{[k_l,n_l]}} = 1, \qquad \forall j\in \bbZ
\end{equation}
which is independent of the scaling sequence. The asymptotically uniform spacing of eigenvalues described by \eqref{eq:Halfline2Clock} or \eqref{eq:WeakClock} is sometimes described as clock spacing, compare \cite{AvilaLastSimon}.

It may seem surprising that two-sided clock behavior is so fragile; the reflectionless condition \eqref{eqnReflectionlessJacobi} is destroyed by the change of a single Jacobi coefficient $a_0$, whereas absolutely continuous spectrum is preserved by trace class perturbations. 

To address this, we will show that clock behavior can indeed be destroyed by the change of a single off-diagonal Jacobi coefficient. We will illustrate this with a very explicit  example. The free Jacobi matrix ($a_n= 1$, $b_n = 0$ for all $n$) is reflectionless with spectrum $[-2,2]$; in fact, it satisfies the condition \eqref{eqnReflectionlessJacobi} pointwise at every $\xi \in (-2,2)$. In particular, it has two-sided clock behavior at $\xi = 0$, as described by Theorem~\ref{thmTwoSidedClock1}. We will show that this behavior is destroyed by changing the coefficient $a_0$.

\begin{example}\label{xmplAlmostFree}
Let $J$ be the full line Jacobi matrix such that $b_n = 0$ for all $n\in\bbZ$ and $a_n = 1$ for all $n\in\bbZ \setminus \{0\}$.
Labelled around the point $\xi=0$, the zeros $\xi_j^{[-n+1,n]}$ satisfy, for every $j\in\bbZ$,
\begin{align}
\lim_{\substack{n\to\infty \\ n-j\text{ odd}}} n \left(\xi_{j+1}^{[-n+1,n]} - \xi_j^{[-n+1,n]} \right) & =  4 \arctan a_0 \label{eqnAlmostFreeOdd} \\
\lim_{\substack{n\to\infty \\ n-j\text{ even}}} n \left(\xi_{j+1}^{[-n+1,n]} - \xi_j^{[-n+1,n]} \right) & = 2\pi - 4 \arctan a_0 \label{eqnAlmostFreeEven}
\end{align}
In particular, for $a_0 \neq 1$, there is no clock spacing at $\xi = 0$.
\end{example}

The main novelty in our proofs is the analysis of CD kernels associated to a movable left endpoint. Whereas prior work required analysis of orthogonal polynomials, our results require control over all formal eigensolutions; we obtain this by using scaling limits of half-line transfer matrices derived in \cite{EichLukSimanek}.

The main novelty in our results is the connection to the reflectionless property. If \eqref{eqnReflectionlessJacobi} holds on a set $A$ of positive Lebesgue measure, the Jacobi matrix is said to be reflectionless on $A$. By work of Kotani \cite{Kotani} and Remling \cite{Remling}, reflectionless property is closely related to a.c.\ spectrum; in particular, for an almost periodic Jacobi matrix, \eqref{eqnReflectionlessJacobi} holds a.e.\ on an essential support of a.c.\ spectrum. Thus, Theorem~\ref{thmTwoSidedClock1} generalizes Avila--Last--Simon \cite{AvilaLastSimon} in the sense that it proves more general clock behavior results for almost periodic Jacobi matrices. Part of the motivation for our work comes precisely from the ergodic setting; there, it is a standard fact that different ways of taking larger and larger truncations lead to the same density of states measure.

The reflectionless property has a central role in inverse spectral theory \cite{SodinYuditskii97}; it also has dynamical and scattering interpretations, compare \cite{BreuerRyckmanSimon,JaksicLandonPanati}. Our work offers a new interpretation.

\medskip

In Section~\ref{sectionFloquet}, we will further discuss the solution $\psi$ and prove Lemma~\ref{lemmaChristoffel}. In Section~\ref{sectionOneSided}, we will introduce kernels corresponding to truncations $J_{[k,n]}$, and prove a scaling limit of such kernels to a sine kernel under the assumptions of Theorem~\ref{thmOneSidedClock}. By a standard application of the Freud-Levin theorem, this will prove Theorem~\ref{thmOneSidedClock}. In Section~\ref{sectionTwoSided}, we will study the case of two endpoints going to opposite infinities. We will encounter a mixed limit kernel which depends on both behaviors at $\pm \infty$. We will see how the reflectionless assumption \eqref{eqnReflectionlessJacobi} leads again to a  sine kernel limit, which will prove Theorem~\ref{thmTwoSidedClock1}. In Section~\ref{sectionOneCoefficientChange}, we prove Example~\ref{xmplAlmostFree}. Finally, in Section~\ref{sectionSchrodinger} we prove Schr\"odinger operator analogs of Theorems~\ref{thmOneSidedClock}, \ref{thmTwoSidedClock1} in Theorems \ref{thmOneSidedClockschrod}, \ref{thmTwoSidedClockschrod}. 

\section{Eigensolution Asymptotics} \label{sectionFloquet}

Assume that \eqref{mplusnormallimit} holds. We consider a solution $\psi$ of \eqref{eqnJacobiRecursion} which satisfies \eqref{eqnFloquetSolution}. By viewing this as the limit of Weyl solutions at energy $\xi + iy$, we see that for any $k \in \bbZ$,
\[
- \frac{ \psi_{k+1}}{a_k \psi_k} = \lim_{y\downarrow 0} m_{+,k}(\xi + i y).
\]
We will write $f \sim g$ if $f/g \to 1$.

\begin{lemma}\label{lemmaFloquetTauProduct}
If \eqref{mplusnormallimit} holds, then
\begin{equation}\label{eqnLemma21}
c_{[0,n]}  \sim  \frac{ \Im \eta }{\pi(1+\lvert \eta\rvert^2)} \sum_{j=0}^{n-1} (p_j(\xi)^2 + q_j(\xi)^2 ), \quad n\to\infty.
\end{equation}
\end{lemma}

\begin{proof}
The statement is independent of normalization, so to verify it, we can renormalize the solution so that $\psi_{0} = -1/a_0$. In this case, $\psi_1 = \eta$ and $\psi_n = q_{n-1}(\xi) + \eta p_{n-1}(\xi)$, which implies
\[
\lvert \psi_j \rvert^2 = q_{j-1}(\xi)^2 + \lvert \eta \rvert^2 p_{j-1}(\xi)^2 + 2 \Re \eta p_{j-1}(\xi) q_{j-1}(\xi).
\]
Denote $\tau_n = \sum_{j=0}^{n-1} (p_j(\xi)^2 + q_j(\xi)^2 )$. By \cite[Theorem 1.5]{EichLukSimanek},
\begin{align}
\frac 1{\tau_n} \sum_{j=0}^{n-1} p_j(\xi)^2 & \to \frac{ 1 }{1+\lvert \eta\rvert^2} \label{eqnELS1} \\
\frac 1{\tau_n} \sum_{j=0}^{n-1} q_j(\xi)^2 & \to \frac{ \lvert \eta\rvert^2 }{1+\lvert \eta\rvert^2}\nonumber \\
\frac 1{\tau_n} \sum_{j=0}^{n-1} p_j(\xi) q_j(\xi) & \to -\frac{ \Re \eta }{1+\lvert \eta\rvert^2} \nonumber
\end{align}
so we obtain
\[
\sum_{j=1}^n \lvert \psi_j \rvert^2 \sim 2\tau_n \frac{ \lvert\eta\rvert^2 - \Re \eta \Re \eta }{ 1+\lvert \eta\rvert^2} = 2 \tau_n \frac{(\Im \eta)^2}{ 1+\lvert \eta\rvert^2}.
\]
Since our normalization yields $W(\ol \psi, \psi) = 2i\Im \eta$, this implies the claim.
\end{proof}

\begin{proof}[Proof of Lemma~\ref{lemmaChristoffel}]
This follows from the limits \eqref{eqnELS1}, \eqref{eqnLemma21} together with the standard identity $\Im\eta = \pi w(\xi)$ for boundary values of $m$-functions.
\end{proof}

For the two-sided setting, let us revise the general definition of the scale $c_{[k,n]}$. In general, we will assume 
existence of normal limits
\[
m_\pm(\xi + i 0) = \eta_\pm \in \bbC_+
\]
and consider eigensolutions $\psi^\pm$ defined by
\[
-\frac{\psi_1^+}{a_0 \psi_0^+} = \eta_+, \quad -\frac{\psi_0^-}{a_0 \psi_1^-} = \eta_-.
\]
In general, we consider $c_{[k,n]}$ to be defined by
\[
c_{[k,n]} = \frac{i}{\pi W(\ol \psi^-, \psi^-)} \sum_{j=\min\{k,0\}+1}^{\min\{n,0\}} \lvert \psi_j^- \rvert^2 + \frac{i}{\pi W(\ol \psi^+, \psi^+)} \sum_{j=\max\{k,0\}+1}^{\max\{n,0\}} \lvert \psi_j^+ \rvert^2.
\]
In the reflectionless case $\ol{\eta_-} = 1 / (a_0^2 \eta_+)$ implies that $\ol{\psi^-} = \psi^+$ so this is equivalent to the formula \eqref{eqnScalingHalfline}.

\section{One-sided scaling limits}\label{sectionOneSided}

Here we will encode the Jacobi recursion as a first-order matrix recursion represented by transfer matrices $T_{[k,n]}(z)$, so that for every solution $u$ of
\begin{equation}\label{eqnJacobiRecursion2}
a_{n-1} u_{n-1} + b_n u_n + a_n u_{n+1} = z u_n
\end{equation}
we have
\begin{equation}\label{eqnMatrixForward}
    T_{[k,n]}(z)\begin{pmatrix}
    u_{k+1}\\-a_{k}u_k
\end{pmatrix}
=\begin{pmatrix}
    u_{n+1}\\-a_{n}u_n
\end{pmatrix}.
\end{equation}
By including the negative signs in the vectors, it is ensured that these transfer matrices, when renormalized by taking $T_{[0,n]}(0)^{-1} T_{[0,n]}(z)$, will correspond to the canonical system which also has Weyl function equal to $m_+$ (see \cite{EichLukSimanek,EichLukWor}). In other words, with the notation
\[
\begin{pmatrix} A_{11} & A_{12} \\ A_{21} & A_{22} \end{pmatrix} * \tau = \frac{A_{11} \tau + A_{12}}{A_{21} \tau + A_{22}},
\]
for any $\tau \in \ol{\bbC_+}= \C_+\cup\R\cup\{\infty\}$, we have
\[
T_{[0,n]}(z)^{-1} * \tau \to m_+(z), \qquad n\to\infty.
\]
Note that
\[
T_{[k,n]}(z)=T_{[0,n]}(z)T_{[0,k]}(z)^{-1}.
\]
Let
\[
j=\begin{pmatrix}
    0&-1\\1&0
\end{pmatrix}.
\]

We denote
\begin{align}\label{eq:2}
    M_{[k,n]}(z)=T_{[0,n]}(\xi)^{-1}T_{[k,n]}(z)T_{[0,k]}(\xi).
\end{align}
Moreover, let us define 
\begin{align}\label{eq:3}
\cK_{[k,n]}(z,w)=\frac{T_{[k,n]}(w)^*jT_{[k,n]}(z)-j}{\overline w-z}.
\end{align}
Since $T_{[0,n]}(\xi)jT_{[0,n]}(\xi)^*=j$, we see that 
\begin{align}\label{eq:4}
\cK_{[k,n]}(z,w)=(T_{[0,k]}^{-1}(\xi))^*\frac{ M_{[k,n]}(w)^{*}jM_{[k,n]}(z)-j}{\overline w-z}T_{[0,k]}^{-1}(\xi).
\end{align}
Denote
\[
\tau_n=\tr \cK_{[0,n]}(\xi,\xi).
\]
If we denote by $p_j$, $q_j$ the orthogonal polynomials and second kind polynomials associated to $J_+$, then $\tau_n = \sum_{j=0}^{n-1} (p_j(\xi)^2 + q_j(\xi)^2)$. 
This is the natural scale for matrix scaling limits; by Lemma~\ref{lemmaFloquetTauProduct}, it is asymptotically proportional to $c_{[0,n]}$.

To study scaling limits of these kernels, we also need to introduce limiting objects. For $\eta\in\bbC_+$ we define 
\[
\mathring{H}_\eta=\frac{1}{1+|\eta|^2}\begin{pmatrix}
    1&-\Re\eta\\-\Re\eta&|\eta|^2
\end{pmatrix}
\]
and 
\[
\M_{t}(z)=e^{ztj\mathring{H}_\eta}.
\]
With
\[
h_\eta=\frac{\Im \eta}{1+|\eta|^2}=\sqrt{\det \mathring{H}_\eta}
\]
we have 
\begin{align*}
    j^{-1}\frac{\mathring{H}_\eta}{h_\eta}j=h_\eta \mathring{H}_\eta^{-1},
\end{align*}
which implies that 
\[
\M_{t}(z)=\cos(tzh_\eta)I+ \frac{\sin(tzh_\eta)}{h_\eta}j\mathring{H}_\eta.
\]
We also define the associated kernel
\[
\K_{t}(z,w)=\frac{\M_{t}(w)^*j\M_{t}(z)-j}{\ol w-z},
\]
which can be written as
\begin{align}\label{eq:KernelFormula}
    \K_{t}(z,w)=\frac{j(\cos(th_\eta(\overline{w}-z))-1)+\frac{\mathring{H}_\eta}{h_\eta}\sin(th_\eta(\overline{w}-z))}{\overline{w}-z}.
\end{align}

We will use the following theorem:

\begin{theorem}[\cite{EichLukSimanek}]\label{thm:ELS1}
 Fix $\xi \in \bbR$ such that \eqref{mplusnormallimit} holds.  Then, uniformly on compact subsets of $\bbC$,
\begin{align*}
    \lim_{n\to\infty} M_{[0,n]}\left(\xi+\frac{z}{\tau_n}\right)= \M_{1}(z).
\end{align*}
In particular, uniformly in compact subsets of $\bbC\times\bbC$,
\[
\lim_{n\to\infty}\frac{\cK_{[0,n]}\left(\xi+\frac{z}{\tau_n},\xi+\frac{w}{\tau_n}\right)}{\tau_n}=\K_{1}(z,w).
\]
\end{theorem}

We will prove scaling limits of CD kernels for $J_{[k,n]}$, defined by
\[
K_{[k,n]}(z,w) = \binom 10^* \frac{T_{[k,n]}(w)^*j T_{[k,n]}(z) - j}{\ol w - z}   \binom 10.
\]
We will repeatedly use the vectors
\begin{align}\label{eq:vk}
v_k=T_{[0,k]}(\xi)^{-1}\begin{pmatrix}
        1 \\ 0
    \end{pmatrix}.
\end{align}
Note that $v_k \in \bbR^2 \setminus \{ \binom 00 \}$. We will first work along suitable subsequences:

\begin{lemma}
Fix $\xi \in \bbR$ such that \eqref{mplusnormallimit} holds. 
Assume that $k_l,n_l$ are such that 
\[
\lim_{l\to\infty}\frac{\tau_{k_l}}{\tau_{n_l}}=\sigma\in [0,1),
\]
 the sequence $k_l$ is bounded below,  $n_l \to \infty$ as $l\to\infty$, and the limit
\begin{align*}
    \lim_{l\to\infty}\frac{v_{k_l}}{\|v_{k_l}\|}
\end{align*}
exists. Then
\begin{align}\label{eq:Halfline2limit}
\lim_{\ell\to\infty}&\frac {K_{[k_l,n_l]}\left(\xi + \frac{z}{h_\eta (\tau_{n_l} - \tau_{k_l})}, \xi + \frac{w}{h_\eta (\tau_{n_l} - \tau_{k_l})} \right)}{(\tau_{n_l}-\tau_{k_l}) v_{k_l}^* \mathring{H}_\eta v_{k_l}}
 = \frac{\sin(\ol w - z)}{\ol w - z}.    
\end{align}
\end{lemma}
\begin{proof}
By shifting, without loss of generality we can assume $k_l \ge 0$.    From \eqref{eq:4} we see that
    \[
    K_{[k_l,n_l]}(z,w)=v_{k_l}^*\frac{ M_{[k_l,n_l]}(w)^{*}jM_{[k_l,n_l]}(z) -j}{\overline w-z}v_{k_l}.
    \]
    Let us write $\tau_l=\tau_{n_l}-\tau_{k_l}$, and note that $\tau_{n_l}/\tau_{l}\to 1/(1-\sigma),$ and $\tau_{k_l}/\tau_{l}\to \sigma/(1-\sigma)$. We denote $
D(k,\eta) = v_k^* \mathring{H}_\eta v_k$. Then we have
    \begin{align*}
        &\frac 1{\tau_{l} D(k_l,\eta)}
K_{[k_l,n_l]}\left(\xi + \frac{z}{\tau_{n_l}-\tau_{k_l}}, \xi + \frac{w}{\tau_{n_l}-\tau_{k_l}} \right)=
\\
& \qquad \frac{\|v_{k_l}\|^2}{D(k_l,\eta)} \frac{v_{k_l}^*}{\|v_{k_l}\|} \frac{ M_{[k_l,n_l]}(\xi + \frac{w}{\tau_{l}})^{*}jM_{[k_l,n_l]}(\xi + \frac{z}{\tau_{l}})-j}{\overline w-z}\frac{v_{k_l}}{\|v_{k_l}\|}.
    \end{align*}
By Theorem \ref{thm:ELS1} we have that 
\[
\lim_{l\to\infty}M_{[0,n_l]}\left(\xi + \frac{z}{\tau_{l}}\right)=M_{[0,n]}\left(\xi + \frac{z}{\tau_{n_l}}\frac{\tau_{n_l}}{\tau_l}\right)=\M_{1/(1-\sigma)}(z)
\]
and likewise 
\[
\lim_{l\to\infty}M_{[0,k_l]}\left(\xi + \frac{z}{\tau_{l}}\right)=\M_{\sigma/(1-\sigma)}(z).
\]
Dividing the first statement by the second, we conclude
\[
\lim_{l\to\infty}M_{[k_l,n_l]}\left(\xi + \frac{z}{\tau_{l}}\right) = 
\M_{1/(1-\sigma)}(z)\M_{\sigma/(1-\sigma)}(z)^{-1}=\M_{1}(z).
\]
Therefore, we conclude that 
\[
\lim_{l\to\infty}\frac{ M_{[k_l,n_l]}(\xi + \frac{w}{\tau_{l}})^{*}jM_{[k_l,n_l]}(\xi + \frac{z}{\tau_{l}}) -j}{\overline w-z}=\K_{1}(z,w).
\]
Moreover, due to convergence of $v_{k_l}/\|v_{k_l}\|$ and 
\[
\frac{D(k_l,\eta)}{\|v_{k_l}\|^2}=\left(\frac{v_{k_l}}{\|v_{k_l}\|}\right)^*\mathring{H}_\eta\frac{v_{k_l}}{\|v_{k_l}\|},
\]
a direct computation using \eqref{eq:KernelFormula} shows that 
\[
\frac{\|v_{k_l}\|^2}{D(k_l,\eta)}\left(\frac{v_{k_l}}{\|v_{k_l}\|}\right)^*\K_{1}(z,w)\frac{v_{k_l}}{\|v_{k_l}\|}=v_{k_l}^*\frac{\K_{1}(z,w)}{D(k_l,\eta)}v_{k_l}=\frac{\sin(h_\eta(z-\overline w))}{h_\eta(z-\ol w)}.
\]
Rescaling the variables $z,w$ by a factor of $h_\eta$ finishes the proof. 
\end{proof}

\begin{proof}[Proof of Theorem~\ref{thmOneSidedClock}]
Lemma~\ref{lemmaFloquetTauProduct} implies that as $l\to\infty$, $h_\eta (\tau_{n_l} - \tau_{k_l}) \sim \pi c_{[k_l,n_l]}$. Meanwhile, \eqref{eq:Halfline2limit} implies
\[
K_{[k_l, n_l]}(\xi,\xi) \sim (\tau_{n_l} - \tau_{k_l}) D(k_l, \eta).
\]
Using these two statements, we can restate the conclusion of the previous lemma as
\begin{equation}\label{eqnSineKernelScalingLimit}
\lim_{\ell\to\infty}\frac {K_{[k_l,n_l]}\left(\xi + \frac{z}{c_{[k_l,n_l]}}, \xi + \frac{w}{c_{[k_l,n_l]}} \right)}{ K_{[k_l,n_l]}(\xi,\xi)}
 = \frac{\sin(\pi(\ol w - z))}{\pi(\ol w - z)}.  
\end{equation}
We have concluded that this holds along a sequence such that $\tau_{k_l} / \tau_{n_l}$ converge to some element of $[0,1)$, and $v_{k_l} / \lVert v_{k_l} \rVert$ converges to some element of $S^1 \subset \bbR^2$. Since the value of the limit is independent of $\sigma$, by compactness of $[0,\limsup \tau_{k_l}/\tau_{n_l}] \times S^1$, the scaling limit \eqref{eqnSineKernelScalingLimit} also holds along our sequence.

The limit \eqref{eq:Halfline2Clock} then follows from the Freud-Levin theorem; see e.g. \cite[Theorem 1.2]{AvilaLastSimon} and \cite[Theorem 10.1]{EichLukWor}.
\end{proof}

\section{Two-sided scaling limits} \label{sectionTwoSided}

In this section we discuss phenomena that depend on the behavior at both $\pm\infty$. Thus, in addition to the transfer matrices $T_{[k,n]}(z)$ defined before, we consider the reversed transfer matrices
$
    \overleftarrow{T}_{[n,k]}(z)
$
acting backwards from solutions at $n$ to solutions at $k$, $n\geq k$, so that 
\begin{equation}\label{eqnMatrixBackward}
\overleftarrow{T}_{[n,k]}(z)\begin{pmatrix}
    u_{n}\\-a_{n}u_{n+1}
\end{pmatrix}=\begin{pmatrix}
    u_{k}\\-a_{k}u_{k+1}
\end{pmatrix}
\end{equation}
for every formal eigensolution $u$ of the Jacobi recursion at energy $z$. With this convention and a renormalization at zero energy, the transfer matrices correspond to  the canonical system which has Weyl function equal to $m_-$. In particular, for any $\tau \in \ol{\bbC_+}$,
\[
\overleftarrow{T}_{[0,k]}(z)^{-1} * \tau \to m_-(z), \qquad k\to -\infty.
\]
By comparing the matrix recursions \eqref{eqnMatrixForward}, \eqref{eqnMatrixBackward}, we see that 
\begin{align*}
    T_{[k,n]}(z)=\begin{pmatrix}
        0&-1/a_n\\-a_n&0
    \end{pmatrix}
    \vl{T}_{[n,k]}(z)^{-1}\begin{pmatrix}
        0& -1/a_k\\-a_k& 0
    \end{pmatrix}.
\end{align*}
We begin with a rescaling limit of the matrix kernel 
\begin{align*}
    \cK_{[k,n]}(z,w)=\frac{T_{[k,n]}(w)^*jT_{[k,n]}(z)-j}{\overline w-z}
\end{align*}

\begin{lemma}\label{lem:scaling}
Assume that $m_\pm$ have normal limits at $\xi$,
\begin{align}
    \label{eq:5}
\lim_{y\to 0}m_{\pm}(\xi + iy)=\eta_{\pm} \in \bbC_+.
\end{align}
Assume that $k_l \to -\infty$, $n_l \to +\infty$ are such that
\begin{align}\label{eq:convergence}
\lim_{l\to\infty} \frac{ c_{[0,n_l]} }{ c_{[k_l, n_l]}} = \sigma \in [0,1]    
\end{align}
exists. Then for
\[
U_k = \begin{pmatrix}
     0&-1/a_{k}\\-a_{k}&0
 \end{pmatrix} \vl T_{[0,k]}(\xi)
 \begin{pmatrix}
     0&-1/a_{0}\\-a_{0}&0
 \end{pmatrix}
\]
and
\[
\mathring{L}(z) = \mathring{M}_{\sigma\pi/h_{\eta_+}}^{\eta_+}(z)\begin{pmatrix}
    0&-1/a_0\\-a_0&0
\end{pmatrix}
 \mathring{M}_{(1-\sigma)\pi/h_{\eta_-}}^{\eta_-}\left(z\right)^{-1}
 \begin{pmatrix}
     0&-1/a_{0}\\-a_{0}&0
 \end{pmatrix}
\]
we have
\begin{align}\label{eqngosnew}
\lim_{ l \to\infty} U_{k_l}^* \frac {\cK_{[k_l,n_l]}\left(\xi + \frac{z}{c_{[k_l,n_l]}}, \xi + \frac{w}{c_{[k_l,n_l]}} \right)}{c_{[k_l,n_l]}}
U_{k_l} = \frac{\mathring{L}(w)^*j \mathring{L}(z) - j}{\ol w-z}.  
\end{align}
\end{lemma}

\begin{proof}
 Note that for $k < 0 \le n$ we have
\begin{align}
T_{[k,n]}(z) & = T_{[0,n]}(z)T_{[k,0]}(z)  \\
& = T_{[0,n]}(z) \begin{pmatrix}
        0&-1/a_0\\-a_0&0
    \end{pmatrix}
    \vl{T}_{[0,k]}(z)^{-1}\begin{pmatrix}
        0& -1/a_k\\-a_k& 0
    \end{pmatrix}.   \label{eq:connection} 
\end{align}

We will introduce the matrices
\begin{align*}
M_n(z) & = T_{[0,n]}(\xi)^{-1}T_{[0,n]}(z), \\
\vl M_{k}(z) & =\vl T_{[0,k]}(\xi)^{-1}\vl T_{[0,k]}(z).
\end{align*}
Then we can represent
\[
T_{[k,n]}(z) U_k = T_{[0,n]}(\xi) L_{k,n}(z)
\]
where
\[
L_{k,n}(z)=M_{n}(z)
\begin{pmatrix}
    0&-1/a_0\\-a_0&0
\end{pmatrix}
\vl M_{k}(z)^{-1}
\begin{pmatrix}
     0&-1/a_{0}\\-a_{0}&0
 \end{pmatrix}.
\]
Using $U_k, T_{[0,n]}(\xi) \in \SL(2,\bbR)$ we conclude 

\begin{align}
 U_k^* \cK_{[k,n]}(z,w) U_k
=
\frac{L_{k,n}(w)^* j L_{k,n}(z) - j}{\overline w-z}. \label{eqnMatrixKernelConjugated}
\end{align}
As in the previous section, 
\begin{align}\label{eq:1}
    \lim_{n\to\infty}M_n(\xi + z/\tau^+_n)=\mathring{M}_{1}^{\eta_+}(z),\quad \lim_{k\to\infty}\vl M_{k}(\xi + z/\tau^-_k)=\mathring{M}_{1}^{\eta_-}(z)
\end{align}
where $\tau_n^\pm$ are the scaling sequences corresponding to $J_\pm$. We will also have to convert this to our scaling sequence by using Lemma~\ref{lemmaFloquetTauProduct},
\begin{align*}
\lim_{l\to\infty} \frac{\tau_{n_l}^+}{c_{[k_l,n_l]}} & = \sigma \lim_{l\to\infty} \frac{\tau_{n_l}^+}{c_{[0,n_l]}} = \frac{\sigma \pi}{h_{\eta_+}}, \\
\lim_{l\to\infty} \frac{\tau_{n_l}^-}{c_{[k_l,n_l]}} & = (1-\sigma) \lim_{l\to\infty} \frac{\tau_{n_l}^-}{c_{[k_l,0]}} = \frac{(1-\sigma)\pi}{h_{\eta_-}}.
\end{align*}
Combining those with \eqref{eq:1}, we find
\[
\lim_{l\to\infty} L_{k_l,n_l}\left(\xi + \frac{z}{c_{[k_l,n_l]}}\right) = \mathring{L}(z),
\]
and combining this with \eqref{eqnMatrixKernelConjugated} proves  \eqref{eqngosnew}.
\end{proof}

The limit kernels in \eqref{eqngosnew} combine the contributions towards $\pm \infty$, as seen from the form of $\mathring{L}$. The parameter $\sigma$ quantifies the influence of each direction; extreme cases are $\sigma=0$ where only $J_-$ affects the limit and $\sigma=1$ where only $J_+$ affects the limit.
 To the best of our knowledge, these kernels have not appeared in the literature before (besides the special case $\sigma=1$ \cite{EichLukSimanek}).

The following lemma allows us to simplify the limit kernels and explains how the reflectionless condition appears in our analysis:

\begin{lemma}\label{lem:comp}
Let $\eta \in \bbC_+$ and $a > 0$ and denote
\[
\zeta = \frac{1}{a^2\overline\eta}.
\]
Then
\begin{align*}
\begin{pmatrix}
    0&1/a\\a&0
\end{pmatrix}
j\frac{\mathring{H}_{\eta}}{h_{\eta}} = - j \frac{\mathring{H}_{ \zeta}}{h_{\zeta}}\begin{pmatrix}
    0&1/a\\a&0
\end{pmatrix}
\end{align*}
and for $s \ge 0$,
\begin{align*}
\begin{pmatrix}
    0&-1/a \\-a &0
\end{pmatrix}
\mathring{M}_{s/h_{\eta}}^{\eta}(z)^{-1}&=\begin{pmatrix}
    0&-1/a \\-a &0
\end{pmatrix}
\mathring{M}_{s/h_{\zeta}}^{\zeta}(z)
\end{align*}
\end{lemma}

\begin{proof}
First we note that $\Im \eta / \lvert \eta \rvert = \Im \zeta / \lvert \zeta\rvert$ so
\[
\frac{h_{\zeta}}{h_{\eta}} = \frac{|\zeta|}{|\eta|}\frac{1+|\eta|^2}{1+|\zeta|^2}.
\]
Then the first claim follows from the calculation
\begin{align*}
    \frac{1}{|\eta|}\begin{pmatrix}
        0&-1\\1&0
    \end{pmatrix}
    \begin{pmatrix}
        0&1/a\\a&0
    \end{pmatrix}
    \begin{pmatrix}
        0&-1\\1&0
    \end{pmatrix}
    \begin{pmatrix}
        1&-\Re \eta\\-\Re\eta &|\eta|^2
    \end{pmatrix}
    \begin{pmatrix}
        0&1/a\\a&0
    \end{pmatrix}
    =
    \frac{1}{|\eta|}
    \begin{pmatrix}
        a^2|\eta|^2 &-\Re\eta \\-\Re\eta &1/a^2       
    \end{pmatrix}\\
    =
    a^2|\eta|
    \begin{pmatrix}
        1&-\frac{\Re\eta}{a^2|\eta|^2}\\-\frac{\Re\eta}{a^2|\eta|^2}&\frac{1}{a^4|\eta|^2}
    \end{pmatrix}
    =a^2|\eta|
    \begin{pmatrix}
         1&-\Re \zeta\\-\Re\zeta &|\zeta|^2
    \end{pmatrix}
\end{align*}
and the second claim follows from this by taking matrix exponentials.
\end{proof}

Using this,
the matrix limit kernel of Lemma~\ref{lem:scaling} can be rewritten as
\[
\mathring{L}(z) = \mathring{M}_{\sigma\pi/h_{\eta_+}}^{\eta_+}(z) 
 \mathring{M}_{(1-\sigma)\pi/h_{\zeta}}^{\zeta}\left(z\right)
\]
where $\zeta = 1 / (a_0^2 \ol{ \eta_-})$. This shows that the limit kernel corresponds to a canonical system with a piecewise constant Hamiltonian with two intervals of constancy. In particular, if $\eta_+ = \zeta$, we obtain a limit kernel corresponding to a constant Hamiltonian, and this is the case in the proof of Theorem~\ref{thmTwoSidedClock1}.

\begin{proof}[Proof of Theorem~\ref{thmTwoSidedClock1}]
Consider the vectors $U_k^{-1} \binom 10$, and denote $r_k = \lVert U_k^{-1} \binom 10 \rVert$. 
First, assume that we are working along a subsequence such that $c_{[0,n_l]} / c_{[k_l,n_l]} \to \sigma$ and that
\begin{equation}\label{eqnAssumedSubseq}
r_{k_l}^{-1} U_{k_l}^{-1} \binom 10 \to \binom{\cos\phi}{\sin\phi}
\end{equation}
for some $\phi$.  Starting from Lemma 4.1, by right multiplication by \eqref{eqnAssumedSubseq} and left multiplication by its adjoint, we conclude
\begin{align}\label{eqngos}
\lim_{ l \to\infty}&\frac {K_{[k_l,n_l]}\left(\xi + \frac{z}{c_{[k_l,n_l]}}, \xi + \frac{w}{c_{[k_l,n_l]}} \right)}{r_{k_l}^2 c_{[k_l,n_l]}}
 = \binom{\cos\phi}{\sin\phi}^* \frac{\mathring{L}(w)^*j\mathring{L}(z)}{\ol w-z}\binom{\cos\phi}{\sin\phi}.
\end{align}
Using Lemma~\ref{lem:comp}, $\mathring{L}(z)$ simplifies as
\begin{align*}
\mathring{L}(z) & =\mathring{M}_{\sigma \pi / h_{\eta_+}}^{\eta_+}\left(z\right)
 \mathring{M}_{(1-\sigma)\pi / h_{\eta_+}}^{\eta_+}\left(z\right)
 = \mathring{M}_{\pi / h_{\eta_+}}^{\eta_+}\left(z\right)
\end{align*}
so \eqref{eqngos} simplifies to
\[
\lim_{ l \to\infty} \frac {K_{[k_l,n_l]}\left(\xi + \frac{z}{c_{[k_l,n_l]}}, \xi + \frac{w}{c_{[k_l,n_l]}} \right)}{r_{k_l}^2 c_{[k_l,n_l]}}
 = \binom{\cos\phi}{\sin\phi}^* \frac{\frac{\mathring{H}_{\eta_+}}{h_{\eta_+}} \sin(\pi (\ol w - z))}{\ol w-z} \binom{\cos\phi}{\sin\phi}.
\]
Renormalizing this limit by its value at $z=w=0$ we conclude
\begin{equation}\label{eqnBulkUniv5}
\lim_{ l \to\infty} \frac {K_{[k_l,n_l]}\left(\xi + \frac{z}{c_{[k_l,n_l]}}, \xi + \frac{w}{c_{[k_l,n_l]}} \right)}{K_{[k_l,n_l]}(\xi,\xi)}
 = \frac{\sin(\pi (\ol w - z))}{\pi(\ol w-z)}.  
\end{equation}
Since this limit is independent of $(\sigma, \phi) \in [0,1] \times (\bbR / 2\pi\bbZ)$, compactness allows us to conclude \eqref{eqnBulkUniv5} along the original sequence of $k_l, n_l$. Now the conclusion about eigenvalue spacing follows again by the Freud--Levin theorem.
\end{proof}

\section{Destroying clock spacing by a rank one perturbation} \label{sectionOneCoefficientChange}

In this section, we will show that clock spacing can be removed by the change of a single off-diagonal Jacobi coefficient. We will illustrate this by proving Example~\ref{xmplAlmostFree}, which is to a great extent  solvable.

The free Jacobi matrix ($a_n= 1$, $b_n = 0$ for all $n$) is reflectionless with spectrum $[-2,2]$; in fact, its half-line Weyl functions are 
\[
m_\pm(z) = \frac{-z+\sqrt{z^2-4}}2
\]
so the condition \eqref{eqnReflectionlessJacobi} is satisfied pointwise at every $\xi \in (-2,2)$. In particular, it has two-sided clock spacing at $\xi = 0$, as described by Theorem~\ref{thmTwoSidedClock1}. We will show that this behavior is destroyed by a change in the coefficient $a_0$.

\begin{proof}[Proof of Example~\ref{xmplAlmostFree}]
For $\lambda \in (-2,2)$, we substitute $\lambda= 2\cos k$, $k\in (0,\pi)$. Solving the eigenfunction equation
\begin{equation}\label{eqnAlmostFree}
J_{[-n+1,n]} u = \lambda u,
\end{equation}
the truncations correspond to boundary conditions $u_{-n} = u_{n+1}=0$, and by solving the recursion \eqref{eqnAlmostFree} on $[-n,-1]$ and on $[2,n+1]$, the eigenfunction must have the form
\[
u_j = \begin{cases} A \sin((j+n)k) & -n+1 \le j \le 0 \\
B \sin((n+1-j)k) & 1 \le j \le n
\end{cases}
\]
By symmetry, $(u_{1 - j})_{j=-n+1}^n$ is also an eigenfunction. By simplicity of the spectrum, we must have $B = \pm A \neq 0$. Finally, evaluating \eqref{eqnAlmostFree} at positions $0$ and $1$, we have two compatibility equations, which are both equivalent to
\[
\sin((n-1)k) \pm a_0 \sin(nk) = 2\cos k \sin(nk).
\]
This rewrites as $\pm a_0 \sin(nk) = \sin((n+1)k)$, so we conclude that
\[
\sigma(J_{[-n+1,n]} ) \cap (-2,2) = \left\{ 2 \cos k\mid k \in (0,\pi), \frac{\sin((n+1)k)}{\sin(nk)} = \pm a_0 \right\}
\]
(the right-hand side is understood as the union over the sets for the two choices of $\pm$ sign).

To study zeros near $\xi=0$, substitute $k = \pi/2 - t/n$. For fixed $t$, this will describe the locations of zeros at $2\cos k = 2 \sin(t/n) \sim  2t/n$ as $n\to\infty$.

Assume for a moment that $n$ is even. Then
\[
\frac{\sin((n+1)k)}{\sin(nk)} = -\frac{\cos((1+1/n)t)}{\sin t} = \pm a_0.
\]
The entire functions
\[
f_n(z) = \cos((1+1/n)z) \pm a_0 \sin z
\]
converge to $f(z) = \cos z \pm a_0 \sin z$ uniformly on compacts, so by Hurwitz's theorem, the zeros of $f_n$ converge to those of $f$ and we conclude
\[
\lim_{\substack{n\to\infty \\ n\text{ even}}} \frac n2 \xi_j^{[-n+1,n]} = 
\begin{cases}  l\pi + \pi/2 - \arctan a_0 & j=2l+1 \\
l\pi - \pi/2 + \arctan a_0 & j = 2l
\end{cases}
\]
Similarly, for $n$ odd, this reduces to solutions of $\tan = \pm a_0$ and we conclude
\[
\lim_{\substack{n\to\infty \\ n\text{ odd}}} \frac n2 \xi_j^{[-n+1,n]} = 
\begin{cases}  l\pi + \arctan a_0 & j=2l+1 \\
l\pi - \arctan a_0 & j = 2l
\end{cases}
\]
In particular, this implies for any fixed $j\in \bbZ$ that \eqref{eqnAlmostFreeOdd}, \eqref{eqnAlmostFreeEven} hold.
\end{proof}

\section{Schr\"odinger operators}\label{sectionSchrodinger}

In the Schr\"odinger setting, we consider the half-line operator
\begin{align}\label{eq:rule}
    H_+=-\frac{d^2}{dx^2}+V(x)
\end{align}
on $L^2([0,\infty))$ for a potential $V\in L_{\loc}^1([0,\infty))$, with Dirichlet boundary condition:
\begin{align*}
    D(H_+)=\{ \psi\in L^2([0,\infty)): \psi,\psi'\in AC_{\loc}([0,\infty)), -\psi''+V\psi\in L^2([0,\infty)),\psi(0)=0\},
\end{align*}
as well as the full line operator $H$ given again by the expression \eqref{eq:rule}, and acting on $L^2(\bbR)$ with domain 
\begin{align*}
    D(H)=\{ \psi\in L^2(\R): \psi,\psi'\in AC_{\loc}(\R), -\psi''+V\psi\in L^2(\R)\}.
\end{align*}
In either case, we define the Dirichlet truncation to an interval $H_{[a,b]}$ by the same expression with domain
\begin{align*}
    D(H_{[a,b]})=\{ \psi\in L^2([a,b]): \psi,\psi'\in AC([a,b]), -\psi''+V\psi \in L^2([a,b]),\psi(a)=\psi(b)=0\}.
\end{align*}

We study the local distribution of eigenvalues of the operator $H_{[a,b]}$ around a point $\xi \in \bbR$: we denote those eigenvalues by
\[
\dots < \xi_{-1}^{[a,b]} < \xi_0^{[a,b]} < \xi \le \xi_{1}^{[a,b]} < \xi_2^{[a,b]} < \dots
\]
We make the following natural modifications of the quantities above, after which the proofs will follow similarly. Let $m_+$ denote the Weyl function associated  with $H_+$. Given a point $\xi\in \R$ so that 
\begin{equation}\label{eqnmnormallimitSchr}
    \lim_{y\downarrow 0}m_+(\xi+iy)=\eta\in \mathbb C_+,
\end{equation}
we take $\psi$ to be a solution to 
\begin{align}\label{eq:eigenvalue}
    -\psi''(x)+V(x)\psi(x)=\xi \psi(x)
\end{align}
satisfying 
\begin{align}\label{normalization}
    \frac{\psi'(0)}{\psi(0)}=\eta.
\end{align}
Then, 
\begin{align*}
    W(\overline{\psi},\psi)=2i\Im(\eta).
\end{align*}
Mirroring the notation introduced above, we define for $a\leq b$,
\begin{align*}
    c_{[a,b]}=\frac{i}{\pi W(\ol \psi,\psi)}\int\limits_{a}^b|\psi(x)|^2dx.
\end{align*}

The fundamental set of solutions at the left endpoint will replace the orthonormal polynomials above; let $\theta(\,\cdot\,,\xi),\phi(\,\cdot\,,\xi)$ be the fundamental solutions of equation \eqref{eq:eigenvalue} normalized at $0$ by
\begin{align}\label{eq:fundamental}
\theta(0,\xi)=1,\quad \theta'(0,\xi)=0,\qquad
\phi(0,\xi)=0,\quad \phi'(0,\xi)=1,
\end{align}
and we define
\[
\tau(L):=\int_0^L \bigl(\theta(x,\xi)^2+\phi(x,\xi)^2\bigr)\,dx.
\]

We set the notation for transfer matrices, 
\begin{align*}
    T_{[0,x]}(z)=\begin{pmatrix}
        \theta'(x,z)&\phi'(x,z)\\\theta(x,z)&\phi(x,z)
    \end{pmatrix}
\end{align*}
and more generally for $a\leq b$, we denote by $T_{[a,b]}$ the transfer matrix propagating solutions $u$ of \eqref{eq:eigenvalue} from $a$ to $b$:
\begin{align*}
    T_{[a,b]}(z)\begin{pmatrix}
    u'(a)\\u(a)
\end{pmatrix}
=\begin{pmatrix}
    u'(b)\\u(b)
\end{pmatrix}.
\end{align*}
We also define the kernel 
\begin{align*}
    \cK_{[a,b]}(z,w)=\frac{T_{[a,b]}(w)^*jT_{[a,b]}(z)-j}{\ol w-z}
\end{align*}
for $a\leq b$.

Finally, we define the quantities as in the discrete setting:
\begin{align}
    v_b=T_{[0,b]}^{-1}(\xi)\begin{pmatrix}
        1\\ 0
    \end{pmatrix}
\end{align}
Then, we have the following continuum analog of Theorem~\ref{thmOneSidedClock}.
\begin{theorem}\label{thmOneSidedClockschrod}
    Assume that at the point $\xi\in\R$ the normal limit \eqref{eqnmnormallimitSchr} holds.
Then, for any sequences $0\le a_l<b_l$ with $b_l\to \infty$ and 
    \begin{align}\label{eq:supassumption}
        \limsup_{l \to\infty}\frac{c_{[0,a_l]}}{c_{[0,b_l]}}<1
    \end{align}
    we have for all $j\in\bbZ$,
    \begin{align}\label{eq:continuumclock}
        \lim_{l\to\infty}c_{[a_l,b_l]}(\xi_{j+1}^{[a_l,b_l]}-\xi_{j}^{[a_l,b_l]})=1.
    \end{align}
\end{theorem}

The proof of this theorem proceeds exactly as in the discrete setting. The only input that differs more than cosmetically is the continuum version of Lemma~\ref{lemmaFloquetTauProduct}, which is readily proven using \cite[Theorem 1.5]{EichLukSimanek}, applicable since this result is proven in the generality of canonical systems. Although it follows in much the same way, we add a proof for the sake of completeness. 
\begin{lemma}\label{lemmaFloquetTauProductSchrod}
Assume that \eqref{eqnBulkUniv5} holds. Then
\[
c_{[0,L]}
\sim
\frac{\Im\eta}{\pi(1+|\eta|^2)}\,\tau(L),
\qquad L\to\infty.
\]
\end{lemma}
\begin{proof}
    The statement is again independent of normalization, and we may freely normalize $\psi$ so that $\psi(0)=1$, $\psi'(0)=\eta $, we have 
    \[
    \psi(x)=\theta(x,\xi)+\eta\phi(x,\xi)
    \]
    so that 
    \[
    |\psi(x)|^2=\theta(x,\xi)^2+|\eta|^2\phi(x,\xi)^2+2\Re(\eta) \theta(x,\xi)\phi(x,\xi).
    \]
    Then, again by \cite[Theorem 1.5]{EichLukSimanek}, along with \cite[Eq. 1.30]{EichLukSimanek}, we have the convergences
    \begin{align*}
        \frac{1}{\tau(L)}\int_0^{L}\phi(x,\xi)^2\,dx & \to \frac{1}{1+|\eta|^2}, \\
        \frac{1}{\tau(L)}\int_0^{L}\theta(x,\xi)^2\,dx & \to \frac{|\eta|^2}{1+|\eta|^2}, \\
        \frac{1}{\tau(L)}\int_0^{L}\theta(x,\xi)\phi(x,\xi)\,dx & \to -\frac{\Re(\eta)}{1+|\eta|^2}.
    \end{align*}
    Thus, we have again 
    \begin{align*}
        \int_{0}^L |\psi|^2\,dx\sim 2\tau(L)\frac{(\Im(\eta))^2}{1+|\eta|^2}
    \end{align*}
    and the result follows after computing $W(\ol\psi,\psi)=2i\Im(\eta)$.
\end{proof}
We now prove our result for two-sided Schr\"odinger operators, in terms of Weyl functions $m_\pm$ associated with the half-line Dirichlet truncations of the whole line Schr\"odinger operator. We prove:
\begin{theorem}\label{thmTwoSidedClockschrod}
Assume that at the point $\xi \in \bbR$,
\begin{equation}\label{eqnReflectionlessSchrod}
 \lim_{y\downarrow 0} m_+(\xi + iy)  =  -\overline{ \lim\limits_{y\downarrow 0} m_-(\xi + iy) } \in \bbC_+.
\end{equation}
Then for any sequences $a_l < b_l$ with $\limsup_{l\to\infty} a_l < \infty$, $\liminf_{l\to\infty} b_l > -\infty$, and $\lim_{l\to\infty} (b_l - a_l) = \infty$,
 we have
\[
\lim_{l \to \infty} c_{[a_l,b_l]}  ( \xi_{j+1}^{[a_l,b_l]} - \xi_{j}^{[a_l,b_l]} ) = 1, \qquad \forall j\in \bbZ.
\]
\end{theorem}

We begin by relating forward and backward transfer matrices. We consider the reversed transfer matrix
$
    \overleftarrow{T}_{[b,a]}(z)
$
acting backwards from solutions at $b$ to solutions at $a$. We use the convention
\[
\overleftarrow{T}_{[b,a]}(z)\begin{pmatrix}
    -u'(b)\\u(b)
\end{pmatrix}=\begin{pmatrix}
    -u'(a)\\u(a)
\end{pmatrix}.
\]
The minus signs can be motivated by how a reflection in the spatial variable affects the derivative; this is consistent with turning $H_-$ into an operator on $L^2((0,\infty))$ with the potential $V(-x)$.
We see that 
\begin{align}\label{eq:051326}
    T_{[a,b]}(z)=\begin{pmatrix}
        -1&0\\0&1
    \end{pmatrix}
    \vl{T}_{[b,a]}(z)^{-1}\begin{pmatrix}
        -1& 0\\0& 1
    \end{pmatrix}.
\end{align}
We adopt the shorthand \[
S=\begin{pmatrix}
    -1&0\\0&1
\end{pmatrix}.
\]
Taking $a< 0\leq b$, we study as before the rescaling limits of the kernel $\cK_{[a,b]}(z,w)$, and we have $T_{[a,b]}(z)=T_{[0,b]}(z)T_{[a,0]}(z)$ and taking $b=0$ in the above formula  
\begin{align*}
   T_{[a,0]}(z)= S
    \vl{T}_{[0,a]}(z)^{-1}S.
\end{align*}
We introduce
\[
M_b(z)=T_{[0,b]}(\xi)^{-1}T_{[0,b]}(z)
\]
and 
\[
\vl{M}_{a}(z)=\vl{T}_{[0,a]}(\xi)^{-1}\vl{T}_{[0,a]}(z)
\]
and write as before 
\[
\eta_\pm =\lim_{y\downarrow0}m_\pm(\xi + iy).
\]
Analogously, we define for $a<0\leq b$,
\begin{align*}
    \tau^+(b)=\int_0^b(\theta(x,\xi)^2+\phi(x,\xi)^2)dx,\; \tau^-(a)=\int_a^0(\theta(x,\xi)^2+\phi(x,\xi)^2)dx
\end{align*}
so that as before we get from Theorem~\ref{thm:ELS1}
\begin{align*}
    \lim_{b\to\infty}M_b(\xi + z/\tau^+(b)) & =\mathring{M}_{[0,1]}^{\eta_+}(z), \\
    \lim_{a\to -\infty}\vl{M}_{a}(\xi + z/\tau^-(a)) & =\mathring{M}_{[0,1]}^{\eta_-}(z).
\end{align*}

With these definitions in hand, and by an entirely analogous proof, we have the following continuum analog of Lemma~\ref{lem:scaling}:
\begin{lemma}\label{lem:scalingschrod}
Assume that $m_\pm$ have normal limits at $\xi$,
\begin{align}
\lim_{y\to 0}m_{\pm}(\xi + iy)=\eta_{\pm} \in \bbC_+.
\end{align}
Assume also that $a_l \to -\infty$, $b_l \to +\infty$ are such that
\begin{align}
\lim_{l\to\infty} \frac{ c_{[0,b_l]} }{ c_{[a_l, b_l]}} = \sigma \in [0,1]    
\end{align}
exists. Then for
\[
U_a = S \vl T_{[0,a]}(\xi)
 S
\]
and
\[
\mathring{L}(z) = \mathring{M}_{\sigma\pi/h_{\eta_+}}^{\eta_+}(z)S
 \mathring{M}_{(1-\sigma)\pi/h_{\eta_-}}^{\eta_-}\left(z\right)^{-1}
 S
\]
we have
\begin{align}\label{eqngosnewSchr}
\lim_{ l \to\infty} U_{a_l}^* \frac {\cK_{[a_l,b_l]}\left(\xi + \frac{z}{c_{[a_l,b_l]}}, \xi + \frac{w}{c_{[a_l,b_l]}} \right)}{c_{[a_l,b_l]}}
U_{a_l} = \frac{\mathring{L}(w)^*j \mathring{L}(z) - j}{\ol w-z}.  
\end{align}
\end{lemma}

As before, we may denote $r_a = \lVert U_{a_l}^{-1} \binom 10\rVert$, and along a subsequence such that we have convergence of the form
\[
\frac{ U_{a_l}^{-1} \binom 10 }{r_{a_l}} \to \binom{\cos\phi}{\sin\phi}
\]
for some $\phi$, Lemma~\ref{lem:scalingschrod} implies a scaling limit of scalar kernels
\begin{align}\label{eqngosSchr}
\lim_{ l \to\infty}&\frac {K_{[a_l,b_l]}\left(\xi + \frac{z}{c_{[a_l,b_l]}}, \xi + \frac{w}{c_{[a_l,b_l]}} \right)}{r_{a_l}^2 c_{[a_l,b_l]}}
 = \binom{\cos\phi}{\sin\phi}^* \frac{\mathring{L}(w)^*j\mathring{L}(z)}{\ol w-z}\binom{\cos\phi}{\sin\phi}.
\end{align}

We now record the continuum analog of Lemma~\ref{lem:comp}, which is significantly simpler in the continuum. In this context, this may be viewed as a manifestation of the reflectionlessness condition being simpler in this setting.

\begin{lemma}
    We have for $\eta\in \C_+$, $t\in \R$,
    \begin{align}\label{eq:conjugation}
        S\mathring{M}_{t}^{\eta}(z)^{-1}S=\mathring{M}^{-\ol \eta}_{t}.
    \end{align}
    and if $\eta_+=-\ol \eta_-$
    \begin{align}\label{eq:reflectioncollapse}
\mathring{M}_{\sigma}^{\eta_+}\left(z\right)
S
 \mathring{M}_{1-\sigma}^{\eta_-}\left(z\right)^{-1} S =\mathring{M}^{\eta_+}_{1}(z) 
    \end{align}
\end{lemma}
\begin{proof}
    For the identity \eqref{eq:conjugation}, it suffices to note that conjugation by $S$ swaps the signs of the off diagonals:
    \begin{align*}
    S\mathring{H}_\eta S=\frac{1}{1+|\eta|^2}\begin{pmatrix}
        1&\Re(\eta)\\
        \Re(\eta)&|\eta|^2
    \end{pmatrix} =\mathring{H}_{-\overline{\eta}}
    \end{align*}
    so that since $SjS=-j$, by taking matrix exponentials we have \eqref{eq:conjugation}.

    The second identity \eqref{eq:reflectioncollapse} follows immediately from the first, since by \eqref{eq:conjugation} we have
    \begin{align*}
\mathring{M}_{\sigma}^{\eta_+}\left(z\right)
S
 \mathring{M}_{1-\sigma}^{\eta_-}\left(z\right)^{-1} S =
 \mathring{M}_{\sigma}^{\eta_+}\left(z\right)
 \mathring{M}_{1-\sigma}^{-\ol \eta_-}\left(z\right) = \mathring{M}^{\eta_+}_{1}(z)
    \end{align*}
    since $\eta_+=-\ol{\eta}_-$.
\end{proof}

With this in hand, as in the previous section, we see that the reflectionless condition leads to a simplification of $\mathring{L}(z)$, with the condition $\eta_+=-\ol \eta_-$ implying that the limiting matrix kernel corresponds to a canonical system with constant (rather than piecewise constant) Hamiltonian. This leads to a $\sigma$-independent simplification of the limit \eqref{eqngosSchr} in the reflectionless case.

\bibliographystyle{plainnat}
\bibliography{lit.bib}

\end{document}